\documentclass[runningheads,a4paper]{llncs}

\usepackage{amssymb}
\usepackage{graphicx}

\usepackage{url}
\newcommand{\keywords}[1]{\par\addvspace\baselineskip\noindent\keywordname\enspace\ignorespaces#1}

\usepackage{hyperref}
\usepackage{tikz}
\usepackage{longtable}%Make a table that takes up more than a single page?
\usepackage{array}%Matrices
\def\emphasize#1{\textit{#1}}

\newcommand{\Z}{\ensuremath{{\mathcal Z}}}
\newcommand{\zn}{\ensuremath{\Z_n}}

\newcommand{\ceil}[1]{\left\lceil #1\right\rceil}
\newcommand{\floor}[1]{\left\lfloor #1\right\rfloor}
\begin{document}
\def\LongTitle{Large Networks of Diameter Two\\Based on Cayley Graphs}
\def\ShortTitle{Large Networks of Diameter Two Based on Cayley Graphs}

\mainmatter  % start of an individual contribution

\title{\LongTitle}

\titlerunning{\ShortTitle} 

\author{Marcel Abas%
%\thanks{Thanks, if you WANT}
}% 
\authorrunning{\ShortTitle}

\institute{Institute of Applied Informatics, Automation and Mechatronics\\
Faculty of Materials Science and Technology in Trnava\\
Slovak University of Technology in Bratislava\\
Trnava, Slovak Republic\\
abas@stuba.sk
}

\toctitle{}
\tocauthor{}
\maketitle

\begin{abstract}
In this contribution we present a construction of large networks of diameter two and of order $\frac{1}{2}d^2$ for every degree $d\geq 8$, based on Cayley graphs with surprisingly simple underlying groups. 
For several small degrees we construct Cayley graphs of diameter two and of order greater than $\frac23$ of Moore bound and   
we show that Cayley graphs of degrees
$d\in\{16,17,18,23,24,31,\dots,35\}$ constructed in this paper are the largest currently known vertex-transitive graphs of diameter two.
\keywords{Degree; Diameter; Moore bound; Cayley graph, Networks.}
\end{abstract}
%%%%%%%%%%%%%%%%%%%%%%%%%%%%%%%%%%%%%%%%%%%%%%%%%%%%%%%%%%%%%%%%%%%%%%%%%%%%%%%%%%%%%%%%%%%%%%%%%%%%%%%%%%%%
%%%%%%%%%%%%%%%%%%%%%%%%%%%%%%%%%%%%%%%%%%%%%%%%%%%%%%%%%%%%%%%%%%%%%%%%%%%%%%%%%%%%%%%%%%%%%%%%%%%%%%%%%%%%
\section{Introduction}\label{sec_introduction}  
 
Nowadays, large-scale networks (interconnection, optical, social, electrical, etc.) are a subject of very intensive study.
Representing nodes of networks by vertices and communication lines by (directed) edges, networks can be modeled by (di)graphs. Below in figure \ref{fig_model} we can see a model of a simple computer network with computers $c_0,c_1,\dots,c_{7}$.

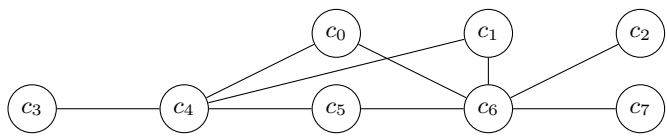
\begin{figure}\label{fig_model} 
\begin{center}
\begin{tikzpicture}
\foreach \i in {0,1,2} \node[draw,circle] (a\i) at (2*\i,0) {$c_\i$};
\foreach \i/\j in {3/2,4/1} \node[draw,circle] (a\i) at (-2*\j,-1) {$c_\i$};
\foreach \i/\j in {5/0,6/1,7/2} \node[draw,circle] (a\i) at (2*\j,-1) {$c_\i$};
\draw (a4)--(a0); \draw (a4)--(a1); 
\draw (a6)--(a5); \draw (a6)--(a0); \draw (a6)--(a1); \draw (a6)--(a2); \draw (a6)--(a7);
\draw (a3)--(a4)--(a5);
\end{tikzpicture}
\caption{A model of a simple computer network}
\end{center}
\end{figure}

Maximum communication delay and maximum communication lines connected to a node are the two main basic limitations on any network. These parameters correspond to the diameter and the maximum (out)degree, respectively, of the corresponding (di)graph. 
The next important property which a good model of a network might posses is simple and efficient routing algorithm. 
Since Cayley graphs are vertex-transitive, it is possible to implement the same routing and communication schemes at each node of the network they model \cite{H1997}.

The problem to find, for given diameter $k$ and maximum degree $d$, the largest order $n(d,k)$ of a graph with given parameters, is in graph theory known as
\emphasize{degree-diameter} problem. There is a well known upper bound on the number $n(d,k)$ - \emphasize{Moore bound}, which gives 
$n(d,k)\leq 1+d+d(d-1)+d(d-1)^2+\cdots+d(d-1)^{k-1}$ for all positive degrees $d$ and diameters $k$. 

The Moore bound for diameter $k=2$ is $n(d,2)\leq d^2+1$ and for degrees $d\geq4$, $d\ne7$ and $d\ne 57$ we have the bound $n(d,2)\leq d^2-1$ \cite{EFH1980}.
The maximum order of a Cayley graph of diameter two and degree $d$ is denoted by $C(d,2)$ and for these graphs we have the following results. In \cite{SS2005} the authors constructed Cayley graphs of diameter two and of order $\frac12(d+1)^2$ for all degrees $d=2q-1$ where $q$ is an odd prime power and the same authors gave a construction of Cayley graphs of diameter two and of order $d^2-O(d^{\frac32})$ for an infinite set of degrees $d$ of a very special type \cite{SS2012}. It was shown in \cite{A2014} that for all degrees $d\geq4$ we have $C(d,2)\geq\frac12 d^2-k$ for $d$ even and $C(d,2)\geq\frac12(d^2+d)-k$ for $d$ odd, where $0\leq k\leq 8$ is an integer depending on the congruence class of $d$ modulo 8. In \cite{E2015} the author has shown that lower as well as upper bounds on the number of vertices of Cayley graphs of diameter two and degree $d$ for underlying dihedral groups are assymptotically $\frac12 d^2$. Finally, in \cite{A2016} the author constructed for all degrees $d\geq 360756$ Cayley graphs of diameter two and of order greater than $0.684d^2$. 

In this paper we give a construction of Cayley graphs of diameter two and of order $\frac{1}{2}d^2$ for every degree $d\geq 8$, with surprisingly simple underlying groups. 
For several small degrees we construct Cayley graphs of diameter two and of order greater than $\frac23$ of Moore bound and we show that Cayley graphs of degrees 
$d\in\{16,17,18,23,24,31,\dots,35\}$ constructed in this paper are the largest currently known vertex-transitive graphs of diameter two.
%%%%%%%%%%%%%%%%%%%%%%%%%%%%%%%%%%%%%%%%%%%%%%%%%%%%%%%%%%%%%%%%%%%%%%%%%%%%%%%%%%%%%%%%%%%%%%%%%%%%%%%%%%%%
%%%%%%%%%%%%%%%%%%%%%%%%%%%%%%%%%%%%%%%%%%%%%%%%%%%%%%%%%%%%%%%%%%%%%%%%%%%%%%%%%%%%%%%%%%%%%%%%%%%%%%%%%%%%
\section{Preliminaries}\label{sec_preliminaries}
For a given finite group $\Gamma$ and a unit-free, inverse-closed generating set $X$ of $\Gamma$, the Cayley graph $G=Cay(\Gamma,X)$ is a graph with vertex set $V(G)=\Gamma$ and with edge set $E(G)=\{\{g,h\}\vert g\in\Gamma, g^{-1}h\in X\}$. Since $X$ is inverse-closed (that is $X=X^{-1}$), for every $g^{-1}h\in X$ we have $h^{-1}g\in X$. Therefore our Cayley graphs are undirected. It is well known that Cayley graphs are vertex transitive. The Cayley graph for underlying group $\Gamma=\Z_6$ and generating set 
$X=\{1,3,5\}$ is shown in figure \ref{fig_cayley} below. The edges corresponding to generators $1$ and $(1)^{-1}=5$ are drawn dashed.

Throughout this paper, an additive cyclic group of order $n$, with elements $\{0,1,\dots,n-1\}$ and identity element $0$, will be denoted by $\zn$. 
Let $\Gamma_n=(\zn\times\zn)\rtimes\Z_2$ be a semidirect product of $\Z_2$ acting on $\zn^2=\zn\times\zn$ such that
the non-identity element of $\Z_2$ interchanges the coordinates of elements of $\zn^2$.
That is $0\in\Z_2: (x,y)\to (x,y)$ and $1\in\Z_2: (x,y)\to (y,x)$. We will write the elements of $\Gamma_n$ as triples $(x,y,i)$ where $x,y\in\zn$ and $i\in\{0,1\}$. The inverse element to $(x_0,x_1,i)$ is $(-x_i,-x_{i+1},i)$ and for the product of two elements of $\Gamma_n$ we have $(x_0,x_1,i)\cdot (y_0,y_1,j)=(x_0+y_i,x_1+y_{i+1},i+j)$, where the indices are taken modulo 2. 

\begin{figure}\label{fig_cayley} 
\begin{center}
\begin{tikzpicture}[yscale=.9, xscale=2, line width=.7pt]
\def\nnn{6}
\newcount\njedna\njedna=\nnn\advance\njedna by -1
\newcount\j
\tikzstyle{kresli}=[draw, circle, fill=white, inner sep=2pt]; 
\foreach \i in {0,1,...,\njedna} {
\node[kresli] (a\i) at (360*\i/\nnn:2cm and 1cm) {\footnotesize\i}; 
\j=\i\advance\j by \nnn \node[kresli] (a\the\j) at (a\i) {\footnotesize\i};
                           }
\foreach \i in {0,1,...,\njedna} {\j=\i\advance\j by 1 \path[draw,dashed] (a\i)--(a\the\j);};
\foreach \i in {0,1,...,\njedna} {\j=\i\advance\j by 3 \path[draw] (a\i)--(a\the\j);};
\end{tikzpicture}
\caption{Cayley graph for underlying group $\Gamma=\Z_6$ and generating set $X=\{1,3,5\}$}
\end{center}
\end{figure}
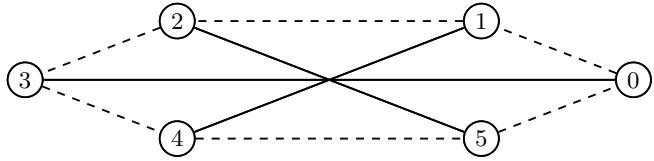
%%%%%%%%%%%%%%%%%%%%%%%%%%%%%%%%%%%%%%%%%%%%%%%%%%%%%%%%%%%%%%%%%%%%%%%%%%%%%%%%%%%%%%%%%%%%%%%%%%%%%%%%%%%%
%%%%%%%%%%%%%%%%%%%%%%%%%%%%%%%%%%%%%%%%%%%%%%%%%%%%%%%%%%%%%%%%%%%%%%%%%%%%%%%%%%%%%%%%%%%%%%%%%%%%%%%%%%%%
%%%%%%%%%%%%%%%%%%%%%%%%%%%%%%%%%%%%%%%%%%%%%%%%%%%%%%%%%%%%%%%%%%%%%%%%%%%%%%%%%%%%%%%%%%%%%%%%%%%%%%%%%%%%
%%%%%%%%%%%%%%%%%%%%%%%%%%%%%%%%%%%%%%%%%%%%%%%%%%%%%%%%%%%%%%%%%%%%%%%%%%%%%%%%%%%%%%%%%%%%%%%%%%%%%%%%%%%%
\section{Large Cayley Graphs of Diameter Two}\label{sec_large_CG}

\begin{theorem}\label{thm_main}
Let $r\geq1$ be an integer, let $s,\epsilon\in\{0,1\}$ and let $n=4r+2s+\epsilon$. Then there exists a Cayley graph of diameter two, degree 
$d=2n-s+\epsilon$ and of order $\frac12(d+s-\epsilon)^2$.
\end{theorem}

\begin{proof}
We set $m=\floor{\frac{n}{2}}=2r+s$. Let the underlying group of the Cayley graph $G=Cay(\Gamma,X)$ be $\Gamma=\Gamma_n=(\zn\times\zn)\rtimes\Z_2$, defined in section \ref{sec_preliminaries} and let the generating set $X$ be the union $X=A\cup B\cup B^{-1}\cup C\cup C^{-1}$, where the sets $A$, $B$ and $C$ are defined as follows:\\
$A=\{a(i)\vert i\in\{0,1,\dots m+2\epsilon-2\}\}$, $a(i)=(i,-i,1)$, $a(i)^{-1}=a(i)$\\
$B=\{b(i)\vert i\in\{1,2,\dots m\}\}$, $b(i)=(0,i,1)$, $b(i)^{-1}=(-i,0,1)$\\ 
$C=\{c(i)\vert i\in\{0,1,\dots r\}\}$, $c(i)=(m-i,i,0)$, $c(i)^{-1}=(-m+i,-i,0)$\\
We can see that $\vert A\vert=m+2\epsilon-1$, $\vert B\cup B^{-1}\vert=n-\epsilon$ and $\vert C\cup C^{-1}\vert=2r+\epsilon+1$.
Therefore the generating set $X$ has order $(m+2\epsilon-1)+(n-\epsilon)+(2r+\epsilon+1)=2n-s+\epsilon$ and the Cayley graph $G$ has order $\vert G\vert=\vert \Gamma\vert=2n^2$.
Since the degree of $G$ is $d=\vert X\vert=2n-s+\epsilon$, the graph $G$ has order $\frac12(d+s-\epsilon)^2$. To show that the Cayley graph has diameter two it is sufficient to show that every element of $\Gamma_n$ is from $X$ or it can be written as a product of two elements from $X$.\\

The rest of the proof is divided into two parts: in part I) we generate elements of the form $(i,j,0)$ and in the part II) we show how to generate elements of the form $(i,j,1)$. 
In the next, all calculations are performed modulo $n$. Note that if $n$ is even then $-m=m$ and if $n$ is odd then 
$-m=\ceil{\frac{n}{2}}=m+\epsilon$.\\\\
I) Generating elements of the form $(i,j,0)$.\\\\
a) Elements of the form $(i,-i,0)$.\\
1) $0\leq i\leq m-2$\\ 
$(i,-i,0)=a(m-2)\cdot a(m-2-i)=(m-2,2-m,1)(m-2-i,-m+2+i,1)$\\
2a) $n$ is odd\\ 
$(m-1,-m+1,0)=a(m)\cdot a(1)=(m,-m,1)(1,-1,1)$\\
$(m,-m,0)=a(m)\cdot a(0)=(m,-m,1)(0,0,1)$\\
2b) $n$ is even\\
$(m-1,-m+1,0)=c(r)^{-1}\cdot c(r-1+s)^{-1}=(m+r,-r,0)(m+r-1+s,-r+1-s,0)$\\
$(m,-m,0)=b(m)\cdot b(m)=(0,m,1)(0,m,1)$ (if $n$ is even then $m=-m$)\\
The other elements of the form $(i,j,0)$ are inverses of the previous.\\\\
b) Elements of the form $(i,j,0)$, $j\ne -i$.\\
All the other elements in this part will be generated as products of generators from $A\cup B\cup B^{-1}$. It is easy to see that if 
$u,v\in\{(i',j',1)\vert i',j'\in\zn\}$ and 
$u\cdot v=(i,j,0)$, then $v\cdot u=(j,i,0)$, $u^{-1}\cdot v^{-1}=(-j,-i,0)$ and $v^{-1}\cdot u^{-1}=(-i,-j,0)$. Therefore it is sufficient to show how to generate 
elements $(i,j,0)$ for $0\leq i\leq m$, $i\leq j<n-1$.\\
1a) $i=0, 1\leq j\leq m-1$\\
$(0,j,0)=b(m)\cdot b(m-j)^{-1}=(0,m,1)(-m+j,0,1)$\\
1b) $i=0, j=m$\\
$(0,m,0)=b(m)\cdot a(0)=(0,m,1)(0,0,1)$\\
1c) $i=0, m+1\leq j\leq n-1$\\
the elements $(0,j,0)$ are inverses of those in 1a) and 1b)\\
2) $1\leq i\leq m, i\leq j\leq m$\\
$(i,j,0)=b(j)\cdot b(i)=(0,j,1)(0,i,1)$\\
3) $1\leq i\leq m-2+\epsilon, m+1\leq j\leq n-1-i$\\
$(i,j,0)=a(i)\cdot b(-i-j)^{-1}=(i,-i,1)(i+j,0,1)$\\\\\\
II) Generating elements of the form $(i,j,1)$. There are exactly $n$ congruence classes $C_0,C_1,\dots,C_{n-1}$ of elements of the form $(i,j,1)$ such that 
$(i,j,1)\in C_k$ if and only if $i+j=k$, $k=0,1,\dots,n-1.$ Since $(i,j,1)\in C_k$ if and only if $(i,j,1)^{-1}\in C_{-k}$, we will do the proof only for $k=0,1,\dots,m.$ 
For fixed $k$ it is sufficient to show that either first or the second coordinate runs from $0$ to $n-1$.\\\\
a) $1\leq k\leq m-1$\\
1) The second coordinate is $0,1,\dots,r$:\\
$(k-j,j,1)=c(j)\cdot b(m-k)^{-1}=(m-j,j,0)(-m+k,0,1)$\\
2) The second coordinate is $r+s,\dots,2r+s$:\\
$(m+\epsilon+k+j,m-j,1)=b(m-k)^{-1}\cdot c(j)=(-(m-k),0,1)(m-j,j,0)$ \\
3) The second coordinate is $2r+s+\epsilon,\dots,32r+s+\epsilon$:\\
$(m+k-j,m+\epsilon+j,1)=b(m+\epsilon-k)^{-1}\cdot c(j)^{-1}=(-(m+\epsilon-k),0,1)(m+\epsilon+j,-j,0)$\\
4) The second coordinate is $3r+2s+\epsilon,\dots,n-1$:\\ 
$(k+j,-j,1)=c(j)^{-1}\cdot b(m+\epsilon-k)^{-1}=(m+\epsilon+j,-j,0)(-(m+\epsilon-k),0,1)$\\\\
b) $k=m$\\
1) The first coordinate is $0,1,\dots,r$:\\
$(i,m-i,1)=a(0)\cdot c(i)=(0,0,1)(m-i,i,0)$\\
2) The first coordinate is $r+s,\dots,2r+s$:\\
$(m-i,i,1)=c(i)\cdot a(0)=(m-i,i,0)(0,0,1)$ \\
3) The first coordinate is $2r+s+\epsilon,\dots,3r+s+\epsilon$:\\
$(m+\epsilon+i,-i,1)=c(i)^{-1}\cdot a(0)=(m+\epsilon+i,-i,0)(0,0,1)$\\
4) The first coordinate is $3r+2s+\epsilon,\dots,n-1$:\\
$(-i,m+\epsilon+i,1)=a(0)\cdot c(i)^{-1}=(0,0,1)(m+\epsilon+i,-i,0)$\\\\
c) $k=0$\\
1a) $n$ is even, the first coordinate is $0,1,\dots,r$:\\
$(i,-i,1)=b(m)\cdot c(i)=(0,m,1)(m-i,i,0)$ \\
2a) $n$ is even, the first coordinate is $r+s,\dots,2r+s$:\\
$(m-i,i,1)=c(i)\cdot b(m)=(m-i,i,0)(0,m,1)$\\
1,2b) $n$ is odd, the first coordinate is $0,1,\dots,2r+s$:\\
$(i,-i,1)=a(i)$\\
3) The first coordinate is $2r+s+\epsilon,\dots,3r+s+\epsilon$:\\
$(-m+i,m-i,1)=b(m)^{-1}\cdot c(i)=(-m,0,1)(m-i,i,0)$ \\
4) The first coordinate is $3r+2s+\epsilon,\dots,n-1$:\\
$(-i,i,1)=b(m)\cdot c(i)^{-1}=(0,m,1)(m+\epsilon+i,-i,0)$
\end{proof}
%%%%%%%%%%%%%%%%%%%%%%%%%%%%%%%%%%%%%%%%%%%%%%%%%%%%%%%%%%%%%%%%%%%%%%%%%%%%%%%%%%%%%%%%%%%%%%%%%%%%%%%%%%%%
%%%%%%%%%%%%%%%%%%%%%%%%%%%%%%%%%%%%%%%%%%%%%%%%%%%%%%%%%%%%%%%%%%%%%%%%%%%%%%%%%%%%%%%%%%%%%%%%%%%%%%%%%%%%
\section{New Record Cayley Graphs of Small Degrees}\label{sec_record_CG}
In this section we construct Cayley graphs of diameter two of large orders for several small degrees. Some of these graphs has order greater than $2/3$ of Moore bound, and even the graph of degree 16 has order greater than $3/4$ of Moore bound. These graphs were found using GAP (Groups,  Algorithms,  Programming - a System for Computational Discrete Algebra \cite{GAP}).

\begin{theorem}\label{thm_record}
There are the following lower bounds on the order of Cayley graphs of degrees $d\in\{16,21,23,28,31,37,40,46,49,54\}$ and diameter 2:
$C(16,2)\geq 200, C(21,2)\geq 288, C(23,2)\geq 392, C(28,2)\geq 512, C(31,2)\geq 648, C(37,2)\geq 800, C(40,2)\geq 968, C(46,2)\geq 1152, C(49,2)\geq 1352, C(54,2)\geq 1568$.
\end{theorem}

\begin{proof}
We present a detailed proof for $d=16$. 
The underlying group of the corresponding Cayley graph is the group $\Gamma=\Gamma_{10}=\Z_{10}^2\rtimes\Z_2$ (as in Theorem \ref{thm_main}, for n=10) and the generating set
$X=A\cup B\cup B^{-1}\cup C\cup C^{-1}$, 
where $A=\{(0,0,1)\}$, 
$B=\{(1,0,1),(1,3,1),(1,7,1),(5,0,1),(5,2,1)\}$ and 
$C=\{(5,0,0),(4,1,0),(3,2,0)\}$. We can see that the order of the Cayley graph $G=Cay(\Gamma,X)$ is $\vert G\vert=\vert\Gamma\vert=200$ and its degree is $d=\vert X\vert=16$. It can be verified by a straightforward calculation that the Cayley graph $G$ has diameter 2. That is, $C(16,2)\geq 200$. For the calculation one can use, for example, the following code (in GAP \cite{GAP}):
\begin{verbatim}
n:=10;;
A:=[ [ 0, 0, 1 ] ];;
B:=[ [ 1, 0, 1 ], [1, 3, 1 ], [ 1, 7, 1], [5, 0, 1 ], [5, 2, 1] ];;
C:=[ [ 5, 0, 0 ], [ 4, 1, 0 ], [ 3, 2, 0 ] ];;
Xgen:=Union(A,B,C);;
for i in B do AddSet(Xgen,[(n-i[2]) mod n,(n-i[1]) mod n,1]); od;
for i in C do AddSet(Xgen,[(n-i[1]) mod n,(n-i[2]) mod n,0]); od;
times := function(x,y)
 if x[3]=0 
  then return [(x[1]+y[1]) mod n, (x[2]+y[2]) mod n, (x[3]+y[3]) mod 2]; 
  else return [(x[1]+y[2]) mod n, (x[2]+y[1]) mod n, (x[3]+y[3]) mod 2]; 
 fi;
end;;
Generate:=[];; 
for x in Xgen do for y in Xgen do AddSet(Generate,times(x,y)); od; od;
Print("Degree d = ",Size(Xgen),"\n");
Print("Number of generated elements = ",Size(Generate));
\end{verbatim}

\noindent
For degrees $d=21,23,28,31,37,40,46,49,54$, the underlying group of the corresponding Cayley graph is the group $\Z_{n}^2\rtimes\Z_2$, with 
$n=12,14,16,18,20,22,24,26,28$, respectively, and the generating set is $X=A\cup B\cup B^{-1}\cup C\cup C^{-1}$, where $C=\{c(i)\vert i\in\{0,1,\dots r\}\}$, $c(i)=(m-i,i,0)$,
$m=\frac{n}{2}$, $r=\floor{\frac{n}{4}}=\floor{\frac{m}{2}}$. For the sets $A$ and $B$ we have:\\

\noindent
%\begin{tabular}{l}\hline
\begin{longtable}{l}\hline
$n=12,\ m=6,\ r=3,\ \vert G\vert=2\cdot 12^2=288,\ d=21$\\
$A=\{(0,0,1),(3,9,1)\}$\\
$B=\{(0,1,1),(0,2,1),(6,9,1),(4,0,1),(5,0,1),(1,5,1)\}$\\\hline
$n=14,\ m=7,\ r=3,\ \vert G\vert=2\cdot 14^2=392,\ d=23$\\
$A=\{(0,0,1),(9,5,1)\}$\\
$B=\{(0,1,1),(0,2,1),(3,0,1),(12,6,1),(5,0,1),(7,13,1),(4,3,1)\}$\\\hline
$n=16,\ m=8,\ r=4,\ \vert G\vert=2\cdot 16^2=512,\ d=28$\\
$A=\{(0,0,1),(1,15,1),(2,14,1)\}$\\
$B=\{(0,1,1),(0,2,1),(11,8,1),(6,14,1),(0,5,1),(9,13,1),(8,15,1),(4,4,1)\}$\\\hline
$n=18,\ m=9,\ r=4,\ \vert G\vert=2\cdot 18^2=648,\ d=31$\\
$A=\{(0,0,1),(6,12,1),(7,11,1),(14,4,1)\}$\\
$B=\{(0,1,1),(2,0,1),(6,15,1),(1,3,1),(8,5,1),(17,7,1),(13,12,1),(11,15,1),$ \\\phantom{$B=\{$}$(6,3,1)\}$\\\hline
$n=20,\ m=10,\ r=5,\ \vert G\vert=2\cdot 20^2=800,\ d=37$\\ 
$A=\{(0,0,1),(1,19,1),(2,18,1),(3,17,1),(10,10,1),(13,7,1)\}$\\
$B=\{(0,1,1),(5,17,1),(0,3,1),(9,15,1),(6,19,1),(17,9,1),(11,16,1),(3,5,1),$ \\\phantom{$B=\{$}$(0,9,1),(3,7,1)\}$\\\hline
$n=22,\ m=11,\ r=5,\ \vert G\vert=2\cdot 22^2=968,\ d=40$\\
$A=\{(0,0,1),(1,21,1),(2,20,1),(3,19,1),(4,18,1),(5,17,1),(11,11,1)\}$\\
$B=\{(0,1,1),(9,15,1),(3,0,1),(11,15,1),(9,18,1),(2,4,1),(19,10,1),(3,5,1),$ \\\phantom{$B=\{$}$(14,17,1),(20,12,1),(11,0,1)\}$\\\hline
$n=24,\ m=12,\ r=6,\ \vert G\vert=2\cdot 24^2=1152,\ d=46$\\
$A=\{(0,0,1),(1,23,1),(2,22,1),(3,21,1),(4,20,1),(5,19,1),(6,18,1),(7,17,1),$ \\\phantom{$A=\{$}$(12,12,1)\}$\\
$B=\{(0,1,1),(0,2,1),(0,3,1),(0,4,1),(0,5,1),(0,6,1),(0,7,1),(16,16,1),$ \\\phantom{$B=\{$}$(14,19,1),(1,9,1),(14,21,1),(12,0,1)\}$\\\hline
$n=26,\ m=13,\ r=6,\ \vert G\vert=2\cdot 26^2=1352,\ d=49$\\
$A=\{(0,0,1),(1,25,1),(2,24,1),(3,23,1),(4,22,1),(5,21,1),(6,20,1),(7,19,1),$ \\\phantom{$A=\{$}$(8,18,1),(13,13,1)\}$\\
$B=\{(0,1,1),(0,2,1),(0,3,1),(0,4,1),(0,5,1),(0,6,1),(0,7,1),(0,8,1),(16,19,1),$ \\\phantom{$B=\{$}$(17,19,1),(3,8,1),(3,9,1),(13,0,1)\}$\\\hline
$n=28,\ m=14,\ r=7,\ \vert G\vert=2\cdot 28^2=1568,\ d=54$\\
$A=\{(0,0,1),(1,27,1),(2,26,1),(3,25,1),(4,24,1),(5,23,1),(6,22,1),(7,21,1),$ \\\phantom{$A=\{$}$(8,20,1),(9,19,1),(14,14,1)\}$\\
$B=\{(0,1,1),(0,2,1),(0,3,1),(0,4,1),(0,5,1),(0,6,1),(0,7,1),(0,8,1),(0,9,1),$ \\\phantom{$B=\{$}$(17,21,1),(18,21,1),(3,9,1),(3,10,1),(14,0,1)\}$\\\hline
\end{longtable}
%\end{tabular}\\\\
which complete the proof of this theorem.
\end{proof}

%%%%%%%%%%%%%%%%%%%%%%%%%%%%%%%%%%%%%%%%%%%%%%%%%%%%%%%
%%%%%%%%%%%%%%%%%%%%%%%%%%%%%%%%%%%%%%%%%%%%%%%%%%%%%%%
%%%%%%%%%%%%%%%%%%%%%%%%%%%%%%%%%%%%%%%%%%%%%%%%%%%%%%%
%%%%%%%%%%%%%%%%%%%%%%%%%%%%%%%%%%%%%%%%%%%%%%%%%%%%%%%
%%%%%%%%%%%%%%%%%%%%%%%%%%%%%%%%%%%%%%%%%%%%%%%%%%%%%%%%%%%%%%%%%%%%%%%%%%%%%%%%%%%%%%%%%%%%%%%%%%%%%%%%%%%%
%%%%%%%%%%%%%%%%%%%%%%%%%%%%%%%%%%%%%%%%%%%%%%%%%%%%%%%%%%%%%%%%%%%%%%%%%%%%%%%%%%%%%%%%%%%%%%%%%%%%%%%%%%%%
\section{Conclusion and Remarks}\label{sec_remarks}

Below, the reader can see a table of present record Cayley graphs of diameter two 
(retrieved from online table \cite{wiki_Cayley_2}), with our results added (bold font). It shows that, for example, for degrees $13\leq d\leq 57$ our construction 
(plus Theorem \ref{thm_record}) gives better results in 34 cases of total 45 degrees. The orders of Cayley graphs for degrees $3\leq d\leq 12$ are not listed - these graphs were found, and shown to be optimal by Marston Conder \cite{wiki_Marston}. Note that until now the largest known vertex-transitive graphs of diameter two have values 
$(d;2)=(16;162),(17;170),(18;192),(23,24;338),(31,\dots,35;578)$ while our results give $(d;2)=(16,17,18;200),(23,24;392),(31,\dots,35;648)$.

%\clearpage
\newcommand{\myhline}{\noalign{\global\arrayrulewidth1pt}\hline
                      \noalign{\global\arrayrulewidth.4pt}}

\def\percent#1{& #1\%}
\def\percentbf#1{& \textbf{#1}\%}
%\begin{table}[h]
\begin{center}
\protect\setlength{\abovecaptionskip}{30pt}
\begin{longtable}{|c||c|c|c|c||c|}\hline
Degree $d$& E. Loz & SS \cite{SS2005} & A \cite{A2014} & NEW & percentage \cr\myhline
%Degree $d$&&&&\cr\hline
13 & 112  &      &      &  \percent{65.88} \cr\hline
14 & 128  &      &      &  \percent{64.97} \cr\hline
15 & 144  &      &      &  \percent{63.71} \cr\hline
16 & 155  &      &      &  \textbf{200} \percentbf{77.82} \cr\hline
17 & 170  &      &      &  \textbf{200} \percentbf{68.96} \cr\hline
18 & 192  &      &      &  \textbf{200} \percentbf{61.53} \cr\hline
19 & 200  &      &      &  \percent{55.24} \cr\hline
20 & 210  &      &      &  \percent{52.36} \cr\hline
21 &   & 242     &      &  \textbf{288} \percentbf{65.15} \cr\hline
22 &   & 242     &      &  \textbf{288} \percentbf{59.38} \cr\hline
23 &   &      & 270     &  \textbf{392} \percentbf{73.96} \cr\hline
24 &   &      & 280     &  \textbf{392} \percentbf{67.93} \cr\hline
25 &   & 338     &      &  \textbf{392} \percentbf{62.61} \cr\hline
26 &   & 338     &      &  \textbf{392} \percentbf{57.90} \cr\hline
27 &   &      & 378     &  \textbf{392} \percentbf{53.69} \cr\hline
28 &   &      & 392     &  \textbf{512} \percentbf{65.22} \cr\hline
29 &   &      & 434     &  \textbf{512} \percentbf{60.80} \cr\hline
30 &   &      & 448     &  \textbf{512} \percentbf{56.82} \cr\hline
31 &   & 512  &         &  \textbf{648} \percentbf{67.35} \cr\hline
32 &   & 512  &         &  \textbf{648} \percentbf{63.21} \cr\hline
33 &   & 512  &         &  \textbf{648} \percentbf{59.44} \cr\hline
34 &   & 512  &         &  \textbf{648} \percentbf{56.00} \cr\hline
35 &   &      & 630     &  \textbf{648} \percentbf{52.85} \cr\hline
36 &   &      & 648     &  \percent{49.96} \cr\hline
37 &   & 722  &         &  \textbf{800} \percentbf{58.39} \cr\hline
38 &   & 722  &         &  \textbf{800} \percentbf{55.36} \cr\hline
39 &   &      & 774     &  \textbf{800} \percentbf{52.56} \cr\hline
40 &   &      & 792     &  \textbf{968} \percentbf{60.46} \cr\hline
41 &   &      & 858     &  \textbf{968} \percentbf{57.55} \cr\hline
42 &   &      & 880     &  \textbf{968} \percentbf{54.84} \cr\hline
43 &   &      & 946     &  \textbf{968} \percentbf{52.32} \cr\hline
44 &   &      & 968     &  \percent{49.97} \cr\hline
45 &   & 1058 &         &  \percent{52.22} \cr\hline
46 &   & 1058 &         &  \textbf{1152} \percentbf{54.41} \cr\hline
47 &   &      & 1122    &  \textbf{1152} \percentbf{52.12} \cr\hline
48 &   &      & 1144    &  \textbf{1152} \percentbf{49.97} \cr\hline
49 &   & 1250 &         &  \textbf{1352} \percentbf{56.28} \cr\hline
50 &   & 1250 &         &  \textbf{1352} \percentbf{54.05} \cr\hline
51 &   &      & 1326    &  \textbf{1352} \percentbf{51.96} \cr\hline
52 &   &      & 1352    &  \percent{49.98} \cr\hline
53 &   & 1458 &         &  \percent{51.88} \cr\hline
54 &   & 1458 &         &  \textbf{1568} \percentbf{53.75} \cr\hline
55 &   &      & 1534    &  \textbf{1568} \percentbf{51.81} \cr\hline
56 &   &      & 1560    &  \textbf{1568} \percentbf{49.98} \cr\hline
57 &   & 1682 &         &  \percent{51.75} \cr\hline
\caption{Online record table of degree-diameter problem for Cayley graphs and our results \cite{wiki_Cayley_2}}
\label{tbl_main}              
\end{longtable} 
\end{center}
%%%%%%%%%%%%%%%%%%%%%%%%%%%%%%%%%%%%%%%%%%%%%%%%%%%%%%%%%%%%%%%%%%%%%%%%%%%%%%%%%%%%%%%%%%%%%%%%%%%%%%%%%%%%
%%%%%%%%%%%%%%%%%%%%%%%%%%%%%%%%%%%%%%%%%%%%%%%%%%%%%%%%%%%%%%%%%%%%%%%%%%%%%%%%%%%%%%%%%%%%%%%%%%%%%%%%%%%%
\section*{Acknowledgements}

%I would like to express my gratitude to the referees for all the valuable and constructive comments.\\

%The research was supported by VEGA Research Grant No. 1/0811/14 and by the ERDF (ITMS: 26220220179). 
%The research was supported by VEGA Research Grant No. 1/0811/14 and by the project ITMS 26220220179 of Operational Programme 'Research \& Development' funded by the ERDF.
The research was supported by VEGA Research Grant No. 1/0811/14 and by 
the Operational Programme 'Research \& Development' funded by the European Regional Development Fund through implementation of the project ITMS 26220220179.
%%%%%%%%%%%%%%%%%%%%%%%%%%%%%%%%%%%%%%%%%%%%%%%%%%%%%%%%%%%%%%%%%%%%%%%%%%%%%%%%%%%%%%%%%%%%%%%%%%%%%%%%%%%%
%%%%%%%%%%%%%%%%%%%%%%%%%%%%%%%%%%%%%%%%%%%%%%%%%%%%%%%%%%%%%%%%%%%%%%%%%%%%%%%%%%%%%%%%%%%%%%%%%%%%%%%%%%%%
%%%%%%%%%%%%%%%%%%%%%%%%%%%%%%%%%%%%%%%%%%%%%%%%%%%%%%%%%%%%%%%%%%%%%%%%%%%%%%%%%%%%%%%%%%%%%%%%%%%%%%%%%%%%
%%%%%%%%%%%%%%%%%%%%%%%%%%%%%%%%%%%%%%%%%%%%%%%%%%%%%%%%%%%%%%%%%%%%%%%%%%%%%%%%%%%%%%%%%%%%%%%%%%%%%%%%%%%%

\end{document}